\documentclass[a4paper, 12pt]{amsart}
\usepackage[utf8]{inputenc}
\usepackage[english]{babel}
\usepackage[T1]{fontenc}
\usepackage{amsmath, amssymb, amsfonts}
\usepackage[all]{xy}
\usepackage{enumerate}
\usepackage{graphicx}
\usepackage{geometry}
\usepackage{hyperref}

\title{Examples of weakly amenable discrete quantum groups}
\author{Amaury Freslon}
\address{Univ. Paris Diderot, Sorbonne Paris Cité, UMR 7586, 8 place FM/13, 75013, Paris, France}
\email{freslon@math.jussieu.fr}
\date{\today}

\theoremstyle{plain}
\newtheorem{thm}{Theorem}[section]
\newtheorem{prop}[thm]{Proposition}
\newtheorem{cor}[thm]{Corollary}
\newtheorem{lem}[thm]{Lemma}

\theoremstyle{definition}
\newtheorem{de}[thm]{Definition}

\theoremstyle{remark}
\newtheorem{rem}[thm]{Remark}

\DeclareMathOperator{\Ad}{Ad}
\DeclareMathOperator{\Id}{Id}
\DeclareMathOperator{\Ir}{Irr}

\DeclareMathOperator{\Mor}{Mor}
\DeclareMathOperator{\Pol}{Pol}
\DeclareMathOperator{\spa}{\overline{span}}
\DeclareMathOperator{\Tr}{Tr}

\renewcommand{\H}{\mathcal{H}}
\renewcommand{\i}{\imath}

\newcommand{\B}{\mathcal{B}}
\newcommand{\C}{\mathbb{C}}
\newcommand{\D}{\Delta}
\newcommand{\F}{\mathbb{F}}
\newcommand{\G}{\mathbb{G}}

\newcommand{\KK}{\mathcal{K}}
\newcommand{\N}{\mathbb{N}}
\newcommand{\R}{\mathbb{R}}

\newcommand{\Z}{\mathbb{Z}}

\newcommand{\h}{\widehat}

\begin{document}

\begin{abstract}
We prove that the the free orthogonal and free unitary quantum groups $\F O_{N}^{+}$ and $\F U_{N}^{+}$ are weakly amenable and that their Cowling-Haagerup constant is equal to $1$. This is achieved by estimating the completely bounded norm of the projections on the coefficients of irreducible representations of their compact duals. An argument of monoidal equivalence then allows us to extend this result to quantum automorphism groups of finite spaces and even yields some examples of weakly amenable non-unimodular discrete quantum groups with the Haagerup property.
\end{abstract}

\maketitle

\section{Introduction}

The \emph{free orthogonal} and \emph{free unitary} quantum groups were constructed by A. Van Daele and S. Wang in \cite{van1996universal, wang1995free}. They are defined as universal C*-algebras generalizing the algebras of continuous functions on the classical orthogonal and unitary groups, together with some additional structure turning them into \emph{compact quantum groups}. From then on, these compact quantum groups have been studied from various points of view : probabilistic, geometric and operator algebraic. In particular, their reduced C*-algebras and von Neumann algebras form interesting classes of examples somehow in the same way as those arising from discrete groups. After the first works of T. Banica \cite{banica1996theorie, banica1997groupe}, it appeared that these operator algebras are closely linked to free group algebras. This link was made more clear by the results of S. Vaes and R. Vergnioux \cite{vaes2007boundary} on exactness and factoriality and those of M. Brannan \cite{brannan2011approximation} on the Haagerup property and the metric approximation property.

Weak amenability for locally compact groups was originally defined by M. Cowling and U. Haagerup in \cite{cowling1989completely} and studied in the context of real simple Lie groups by J. de Cannière, M. Cowling and U. Haagerup in \cite{cowling1989completely, de1985multipliers}. In the discrete setting, many examples were provided by N. Ozawa's result \cite{ozawa2007weak} stating that all Gromov hyperbolic groups are weakly amenable. Weak amenability has recently attracted a lot of attention since it is a key ingredient in some of S. Popa's deformation/rigidity techniques, see for example \cite{ozawa2010class, ozawa2010classbis}. Another feature of this approximation property is that it provides a numerical invariant which carries to the associated operator algebras and may thus give a way to distinguish them. An introduction to approximation properties for classical discrete groups can be found in \cite[Chapter 12]{brown2008finite}, though no knowledge on this subject will be required afterwards.

The results mentionned in the first paragraph naturally raise the issue of weak amenability for free quantum groups. It has been strongly suspected for some time that they have a Cowling-Haagerup constant equal to $1$, and this is what we prove in the present paper. To do this, we show that the completely bounded norm of the projections on coefficients of a fixed irreducible representation (i.e. on "words of fixed length") grows polynomially. This fact can then easily be combined with M. Brannan's proof of the Haagerup property to yield weak amenability when $F = \Id$, or more generally when $F$ is unitary. In the other (non-unimodular) cases, we are unable to give a complete answer to the problem, due to the lack of a proof of the Haagerup property.

Let us now briefly outline the organization of the paper. In Section \ref{sec:preliminaries}, we recall some basic facts about compact and discrete quantum groups and we fix notations. We also give some fundamental definitions and results concerning free quantum groups. We then introduce weak amenability for discrete quantum groups in Section \ref{sec:quantumwa}. Section \ref{subsec:blocks} contains the first technical part of our result, reducing the problem to controlling the norms of certain blocks of "operator-valued functions" on the discrete quantum groups considered. This can be interpreted as an analogue of the operator-valued Haagerup inequality for free groups. Another technical result is worked out in Section \ref{subsec:recursion} to obtain a suitable bound on the completely bounded norm of the projection on some fixed irreducible representation. Combining these results then easily yields our main result in Section \ref{subsec:final}. Using monoidal equivalence, we can then transfer our estimate to quantum automorphism groups of finite-dimensional C*-algebras and we can then extend our weak amenability result to this part of this class.

\subsection*{Acknowledgments}

We are deeply indebted to R. Vergnioux for the time he spent discussing the arguments of this paper and the knowledge on free quantum groups he was kind enough to share. We would also like to thank M. Brannan, M. de la Salle, P. Fima, P. Jolissaint and S. Vaes for discussions on topics linked to quantum groups and approximation properties at various stages of this project.

\section{Preliminaries}\label{sec:preliminaries}

\subsection{Notations}

All inner products will be taken linear in the second variable. For two Hilbert spaces $H$ and $K$, $\B(H, K)$ will denote the set of bounded linear maps from $H$ to $K$ and $\B(H):=\B(H, H)$. In the same way we will use the notations $\KK(H, K)$ and $\KK(H)$ for compact linear maps. We will denote by $\B(H)_{*}$ the predual of $\B(H)$, i.e. the Banach space of all normal linear forms on $\B(H)$. On any tensor product $H\otimes H'$ of Hilbert spaces, we define the flip operator
\begin{equation*}
\Sigma : \left\{\begin{array}{ccc}
H\otimes H' & \rightarrow & H'\otimes H \\
x\otimes y & \mapsto & y\otimes x
\end{array}\right.
\end{equation*}
We will use the usual leg-numbering notations : for an operator $X$ acting on a tensor product, we set $X_{12}:=X\otimes1$, $X_{23}:=1\otimes X$ and $X_{13}:=(\Sigma\otimes 1)(1\otimes X)(\Sigma\otimes 1)$. The identity map of an algebra $A$ will be denoted $\i_{A}$ or simply $\i$ if there is no possible confusion. For a subset $B$ of a topological vector space $C$, $\spa B$ will denote the \emph{closed linear span} of $B$ in $C$. The symbol $\otimes$ will denote the \emph{minimal} (or spatial) tensor product of C*-algebras or the topological tensor product of Hilbert spaces. The spatial tensor product of von Neumann algebras will be denoted $\overline{\otimes}$ and the algebraic tensor product (over $\C$) will be denoted $\odot$.

\subsection{Compact and discrete quantum groups}

Discrete quantum groups will be seen as duals of compact quantum groups in the sense of Woronowicz. We briefly review the basic theory of compact quantum groups as introduced in \cite{woronowicz1995compact}. Another survey, encompassing the non-separable case, can be found in \cite{maes1998notes}. Emphasis has been put on the explicit description of the associated $L^{2}$-space since this will prove crucial in the sequel.

\begin{de}
A \emph{compact quantum group} $\G$ is a pair $(C(\G), \Delta)$ where $C(\G)$ is a unital C*-algebra and $\Delta : C(\G)\rightarrow C(\G)\otimes C(\G)$ is a unital $*$-homomorphism such that
\begin{eqnarray*}
(\Delta\otimes \i)\circ\Delta & = & (\i\otimes\Delta)\circ\Delta \\
\spa\{\Delta(C(\G))(1\otimes C(\G))\} & = & C(\G)\otimes C(\G) \\
\spa\{\Delta(C(\G))(C(\G)\otimes 1)\} & = & C(\G)\otimes C(\G)
\end{eqnarray*}
\end{de}

The main feature of compact quantum groups is the existence of a Haar state which is both left and right invariant (see \cite[Thm 1.3]{woronowicz1995compact}).

\begin{thm}[Woronowicz]
Let $\G$ be a compact quantum group. There is a unique \emph{Haar state} on $\G$, that is to say a state $h$ on $C(\G)$ such that for all $a\in C(\G)$,
\begin{eqnarray*}
(\i\otimes h)\circ \D(a) = h(a).1 \\
(h\otimes \i)\circ \D(a) = h(a).1
\end{eqnarray*}
\end{thm}

Let $(L^{2}(\G), \pi_{h}, \xi_{h})$ be the associated GNS construction and let $C_{\text{red}}(\G)$ be the image of $C(\G)$ under the GNS map $\pi_{h}$. It is called the \emph{reduced C*-algebra} of $\G$. Let $W$ be the unique unitary operator on $L^{2}(\G)\otimes L^{2}(\G)$ such that
\begin{equation*}
W^{*}(\xi\otimes \pi_{h}(a)\xi_{h}) = (\pi_{h}\otimes \pi_{h})\circ\Delta(a)(\xi\otimes \xi_{h})
\end{equation*} 
for $\xi \in L^{2}(\G)$ and $a\in C(\G)$, and let $\h{W} := \Sigma W^{*}\Sigma$. Then $W$ is a \emph{multiplicative unitary} in the sense of \cite{baaj1993unitaires}, i.e. $W_{12}W_{13}W_{23} = W_{23}W_{12}$ and we have the following equalities :
\begin{equation*}
C_{\text{red}}(\G) = \spa\{(\i\otimes \B(L^{2}(\G))_{*})(W)\} \text{ and } \Delta(x) = W^{*}(1\otimes x)W.
\end{equation*}
Moreover, we can define the \emph{dual discrete quantum group} $\h{\G} = (C_{0}(\h{\G}), \h{\Delta})$ by
\begin{equation*}
C_{0}(\h{\G}) = \spa\{(\B(L^{2}(\G))_{*}\otimes\i)(W)\} \text{ and } \h{\Delta}(x) = \Sigma W(x\otimes 1)W^{*}\Sigma.
\end{equation*}
The two von Neumann algebras associated to these quantum groups are then
\begin{equation*}
L^{\infty}(\G) = C_{\text{red}}(\G)'' \text{ and } \ell^{\infty}(\h{\G})=C_{0}(\h{\G})''
\end{equation*}
where the bicommutants are taken in $\B(L^{2}(\G))$. The coproducts extend to normal maps on these von Neumann algebras and one can prove that $W\in L^{\infty}(\G)\overline{\otimes} \ell^{\infty}(\h{\G})$. The Haar state of $\G$ extends to a state on $L^{\infty}(\G)$.

\subsection{Irreducible representations and the GNS construction}\label{subsec:gns}

We will need in the sequel an explicit description of the GNS construction of the Haar state $h$ using the following notion of irreducible representation of a compact quantum group.

\begin{de}
A \emph{representation} of a compact quantum group $\G$ on a Hilbert space $H$ is an operator $u\in L^{\infty}(\G)\overline{\otimes} \B(H)$ such that $(\D\otimes \i)(u) = u_{13}u_{23}$. It is said to be \emph{unitary} if the operator $u$ is unitary.
\end{de}

\begin{de}
Let $\G$ be a compact quantum group and let $u$ and $v$ be two representations of $\G$ on Hilbert spaces $H_{u}$ and $H_{v}$ respectively. An \emph{intertwiner} (or \emph{morphism}) between $u$ and $v$ is a map $T\in \B(H_{u}, H_{v})$ such that $v(1\otimes T) = (1\otimes T)u$. The set of intertwiners between $u$ and $v$ will be denoted $\Mor(u, v)$.
\end{de}

A representation $u$ will be said to be \emph{irreducible} if $\Mor(u, u) = \C.\Id$ and it will be said to be a \emph{subrepresentation} of $v$ if there is an isometric intertwiner between $u$ and $v$. We will say that two representations are \emph{equivalent} (resp. \emph{unitarily equivalent}) if there is an intertwiner between them which is an isomorphism (resp. a unitary). Let us define two fundamental operations on representations.

\begin{de}
Let $\G$ be a compact quantum group and let $u$ and $v$ be two representations of $\G$ on Hilbert spaces $H_{u}$ and $H_{v}$ respectively. The \emph{direct sum} of $u$ and $v$ is the diagonal sum of the operators $u$ and $v$ seen as an element of $L^{\infty}(\G)\otimes \B(H_{u}\oplus H_{v})$. It is a representation denoted $u\oplus v$. The \emph{tensor product} of $u$ and $v$ is the element $u_{12}v_{13}\in L^{\infty}(\G)\otimes \B(H_{u}\otimes H_{v})$. It is a representation denoted $u\otimes v$.
\end{de}

The following generalization of the classical Peter-Weyl theory holds (see \cite[Section 6]{woronowicz1995compact}).

\begin{thm}[Woronowicz]
Every representation of a compact quantum group is equivalent to a unitary one. Every irreducible representation of a compact quantum group is finite dimensional and every unitary representation is unitarily equivalent to a sum of irreducible ones. Moreover, the linear span of the coefficients of all irreducible representations is a dense Hopf $*$-subalgebra of $C(\G)$ denoted $\Pol(\G)$.
\end{thm}

Let $\Ir(\G)$ be the set of isomorphism classes of irreducible unitary representations of $\G$. If $\alpha\in \Ir(\G)$, we will denote by $u^{\alpha}$ a representative of the class $\alpha$ and by $H_{\alpha}$ the finite dimensional Hilbert space on which $u^{\alpha}$ acts. There are isomorphisms
\begin{equation*}
C_{0}(\h{\G}) = \bigoplus_{\alpha\in \Ir(\G)}\B(H_{\alpha})\text{ and }\ell^{\infty}(\h{\G}) = \prod_{\alpha\in \Ir(\G)}\B(H_{\alpha}).
\end{equation*}
The minimal central projection in $\ell^{\infty}(\h{\G})$ corresponding to the identity of $\B(H_{\alpha})$ will be denoted $p_{\alpha}$.

We now proceed to describe explicitely the GNS representation of the Haar state using the irreducible representations. For any $\alpha\in \Ir(\G)$, there is a unique (up to unitary equivalence) irreducible representation, called the \emph{contragredient representation} of $\alpha$ and denoted $\overline{\alpha}$, such that $\Mor(\varepsilon, \alpha\otimes \overline{\alpha})\neq \{0\} \neq \Mor(\varepsilon, \overline{\alpha}\otimes \alpha)$, $\varepsilon$ denoting the trivial representation. This yields an antilinear isomorphism
\begin{equation*}
j_{\alpha} : H_{\alpha}\rightarrow H_{\overline{\alpha}}.
\end{equation*}
The matrix $j_{\alpha}^{*}j_{\alpha}\in \B(H_{\alpha})$ is unique up to multiplication by a real number. We will say that $j_{\alpha}$ is \emph{normalized} if $\Tr(j_{\alpha}^{*}j_{\alpha}) = \Tr((j_{\alpha}^{*}j_{\alpha})^{-1})$ (this only determines $j_{\alpha}$ up to some complex number of modulus one, but this is of no consequence in our context). In that case we will set $Q_{\alpha} = j_{\alpha}^{*}j_{\alpha}$, $\dim_{q}(u^{\alpha}) = \Tr(Q_{\alpha}) = \Tr(Q_{\alpha}^{-1})$ and $t_{\alpha}(1) = \sum j_{\alpha}(e_{i})\otimes e_{i}$, where $(e_{i})$ is some fixed orthonormal basis of $H_{\alpha}$. We will also set $u^{\alpha}_{i, j} = (\i\otimes e_{i}^{*})u^{\alpha}(\i\otimes e_{j})$. Note that by construction, $t_{\alpha} : \C\rightarrow H_{\overline{\alpha}}\otimes H_{\alpha}$ is a morphism of representations. Let us define a map
\begin{equation*}
\psi_{\alpha} : \left\{\begin{array}{ccc}
H_{\overline{\alpha}}\otimes H_{\alpha} & \rightarrow & C_{\text{red}}(\G) \\
\eta\otimes \xi & \mapsto & \pi_{h}[(1 \otimes j_{\overline{\alpha}}(\eta)^{*})u^{\alpha}(1 \otimes\xi)]
\end{array}\right.
\end{equation*}
According to \cite[Eq. 6.8]{woronowicz1995compact} we have, for any $z, z'\in H_{\overline{\alpha}}\otimes H_{\alpha}$,
\begin{equation*}
h(\psi_{\alpha}(z')^{*}\psi_{\alpha}(z)) = \frac{1}{\dim_{q}(\alpha)}\langle z', z\rangle.
\end{equation*}
and $\Psi = \oplus_{\alpha}\sqrt{\dim_{q}(\alpha)}\psi_{\alpha}.\xi_{h} : \oplus_{\alpha} (H_{\overline{\alpha}}\otimes H_{\alpha}) \rightarrow L^{2}(\G)$ is an isometric isomorphism of Hilbert spaces. If we let $E_{i, j}$ denote the operator on $H_{\alpha}$ sending $e_{i}$ to $e_{j}$ and the other vectors of the basis to $0$, we can define another map
\begin{equation*}
\Phi_{\alpha} : \left\{\begin{array}{ccc}
H_{\overline{\alpha}}\otimes H_{\alpha} & \longrightarrow & \B(H_{\alpha}) \\
j_{\overline{\alpha}}(e_{i})\otimes e_{j} & \mapsto & E_{i, j}
\end{array}\right.
\end{equation*}
Now, we observe that $\Theta_{\alpha} = \psi_{\alpha}\circ\Phi_{\alpha}^{-1} : \B(H_{\alpha})\rightarrow C_{\text{red}}(\G)$ sends $E_{i, j}$ to $\pi_{h}(u^{\alpha}_{i, j})$ and that
\begin{eqnarray*}
h(\Theta_{\alpha}(E_{i, j})^{*}\Theta_{\alpha}(E_{k, l})) & = & \frac{1}{\dim_{q}(\alpha)}\langle (\Phi_{\alpha}^{-1}(E_{i, j}), \Phi_{\alpha}^{-1}(E_{k, l})\rangle \\
& = & \frac{1}{\dim_{q}(\alpha)}\langle j_{\alpha}(e_{i})\otimes e_{j}, j_{\alpha}(e_{k})\otimes e_{l}\rangle \\
& = & \frac{\delta_{j, l}}{\dim_{q}(\alpha)}\langle Q_{\alpha}e_{i}, e_{k}\rangle \\
& = & \frac{1}{\dim_{q}(\alpha)}\Tr(Q_{\alpha}E_{i, j}^{*}E_{k, l}).
\end{eqnarray*}
Thus, if we endow $\B(H_{\alpha})$ with the scalar product $\langle A, B\rangle_{\alpha} = \dim_{q}(\alpha)^{-1}\Tr(Q_{\alpha}A^{*}B)$, we get an isometric isomorphism of Hilbert spaces
\begin{equation*}
\Theta = \oplus_{\alpha} \Theta_{\alpha}.\xi_{h} : \oplus_{\alpha} \B(H_{\alpha}) \rightarrow L^{2}(\G).
\end{equation*}
Note that the duality map $S_{\alpha} : A \mapsto \langle A, .\rangle_{\alpha}$ being bijective on the finite dimensional space $\B(H_{\alpha})$, one can endow $\oplus_{\alpha} \B(H_{\alpha})_{*}$ with a Hilbert space structure making it isomorphic to $L^{2}(\G)$ via $\Theta\circ (\oplus_{\alpha}S_{\alpha}^{-1})$. This isomorphism is "natural" since it sends $\omega\in \B(H_{\alpha})_{*}$ to $\pi_{h}[(\i\otimes\omega)(u^{\alpha})].\xi_{h}$.

Let $u^{\alpha}$ and $u^{\beta}$ be two irreducible representations of $\G$ and assume, for the sake of simplicity, that \emph{every irreducible subrepresentation of $u^{\alpha}\otimes u^{\beta}$ appears with multiplicity one}. This is no restriction in the case of free quantum groups that we will be considering (see Theorem \ref{thm:freefusion}). Let $v_{\gamma}^{\alpha, \beta} : H_{\gamma}\rightarrow H_{\alpha}\otimes H_{\beta}$ be an isometric intertwiner. Note that $v_{\gamma}^{\alpha, \beta}Q_{\gamma} = (Q_{\alpha}\otimes Q_{\beta}) v_{\gamma}^{\alpha, \beta}$. We have,
\begin{eqnarray*}
(\i\otimes \omega)(u^{\alpha})(\i\otimes \omega')(u^{\beta}) & = & (\i\otimes \omega\otimes \omega')(u^{\alpha}_{12}u^{\beta}_{13}) \\
& = & (\i\otimes \omega\otimes \omega')(u^{\alpha}\otimes u^{\beta}) \\
& = & (\i\otimes \omega\otimes \omega')(\sum_{\gamma\subset \alpha\otimes \beta}(\i\otimes v_{\gamma}^{\alpha, \beta})u^{\gamma}(\i\otimes v_{\gamma}^{\alpha, \beta})^{*}) \\
& = & \sum_{\gamma\subset \alpha\otimes \beta}(\i\otimes[\omega\otimes \omega']^{\gamma})(u^{\gamma})
\end{eqnarray*}
where $\omega^{\gamma}(x) = \omega(v_{\gamma}^{\alpha, \beta}\circ x\circ(v_{\gamma}^{\alpha, \beta})^{*})$ for $\omega\in \B(H_{\alpha}\otimes H_{\beta})_{*}$. Using the duality map $S_{\alpha}^{-1}$, we can write the map induced on $C_{\text{red}}(\G)$ by the product under our identification : for $A\in \B(H_{\alpha})$ and $B\in \B(H_{\beta})$,
\begin{equation*}
\Theta_{\alpha}(A).\Theta_{\beta}(B) = \sum_{\gamma\subset \alpha\otimes \beta} \Theta_{\gamma}((v_{\gamma}^{\alpha, \beta})^{*}(A\otimes B)v_{\gamma}^{\alpha, \beta}).
\end{equation*}

We can now give an explicit formula for the GNS representation $\pi_{h}$. Let $x$ be a coefficient of $u^{\alpha}$ and let $\xi\in p_{\beta}L^{2}(\G)\simeq \B(H_{\beta})$. Identify $x$ with $\pi_{h}(x)\xi_{h}$, which is an element of $p_{\alpha}L^{2}(\G)\simeq \B(H_{\alpha})$. Making the identification by $\Theta$ implicit, we have
\begin{equation}\label{eq:gns}
\pi_{h}(x)\xi = \sum_{\gamma\subset\alpha\otimes \beta} (v_{\gamma}^{\alpha, \beta})^{*}(x\otimes \xi)v_{\gamma}^{\alpha, \beta} = \sum_{\gamma\subset\alpha\otimes \beta} \Ad(v_{\gamma}^{\alpha, \beta})(x\otimes \xi).
\end{equation}
 
\subsection{Free quantum groups}

We will be concerned in the sequel with the free unitary and free orthogonal quantum groups. They were first defined by A. Van Daele and S. Wang in \cite{van1996universal, wang1995free} and the definition was later slightly modified by T. Banica in \cite{banica1996theorie}. This section is devoted to briefly recalling the definition and main properties of these free quantum groups. If $A$ is a C*-algebra and if $u = (u_{i, j})$ is a matrix with coefficients in $A$, we set $\overline{u} = (u_{i, j}^{*})$.

\begin{de}\label{de:freeqgroups}
Let $N\in \N$ and let $F\in GL_{N}(\C)$ be such that $F\overline{F} \in \R.\Id$. We denote by $A_{u}(F)$ the universal unital C*-algebra generated by $N^{2}$ elements $(u_{i, j})$ such that the matrices $u=(u_{i, j})$ and $F\overline{u}F^{-1}$ are \emph{unitary}. Similarly, we 
denote by $A_{o}(F)$ the universal unital C*-algebra generated by $N^{2}$ elements $(v_{i, j})$ such that the matrix $v=(v_{i, j})$ is \emph{unitary} and $v=F\overline{v}F^{-1}$.
\end{de}

One can easily check that there is a unique coproduct $\D_{u}$ (resp. $\D_{o}$) on $A_{u}(F)$ (resp. $A_{o}(F)$) such that for all $i, j$,
\begin{eqnarray*}
\D_{u}(u_{i, j}) = \sum_{i, j = 0}^{N}u_{i, k}\otimes u_{k, j} \\
\D_{o}(v_{i, j}) = \sum_{i, j = 0}^{N}v_{i, k}\otimes v_{k, j}
\end{eqnarray*}

\begin{de}
A pair $(A_{u}(F), \D_{u})$ is called a \emph{free unitary quantum group} and will be denoted $U^{+}(F)$. A pair $(A_{o}(F), \D_{o})$ is called a \emph{free orthogonal quantum group} and will be denoted $O^{+}(F)$. Their discrete duals will be denoted respectively $\F U^{+}(F)$ and $\F O^{+}(F)$.
\end{de}

\begin{rem}
The restriction on the matrix $F$ in the definition is equivalent to requiring the fundamental representation $v$ of $O^{+}(F)$ to be irreducible. That assumption is necessary in order to get a nice description of the representation theory of $O^{+}(F)$.
\end{rem}

Any \emph{compact matrix pseudogroup} in the sense of \cite[Def. 1.1]{woronowicz1987compact} is a compact quantum subgroup of a free unitary quantum group. Moreover, if its fundamental corepresentation is equivalent to its contragredient, then it is a compact quantum subgroup of a free orthogonal quantum group. In this sense, we can see $U^{+}(F)$ and $O^{+}(F)$ as quantum generalizations of the usual unitary and orthogonal groups. The representation theory of free orthogonal quantum groups was computed by T. Banica in \cite{banica1996theorie}.

\begin{thm}[Banica]\label{thm:freefusion}
The equivalence classes of irreducible representations of $O^{+}(F)$ are indexed by the set $\N$ of integers ($u^{0}$ being the trivial representation and $u^{1}=u$ the fundamental one), each one is isomorphic to its contragredient and the tensor product is given (inductively) by
\begin{equation*}
u^{1}\otimes u^{n} = u^{n+1}\oplus u^{n-1}.
\end{equation*}
Moreover, if $N=2$, then $\dim_{q}(u^{n})=n+1$. Otherwise,
\begin{equation*}
\dim_{q}(u^{n}) = \frac{q^{n+1} - q^{-n-1}}{q - q^{-1}},
\end{equation*}
where $q + q^{-1}= \Tr(Q_{1})$ and $0\leqslant q\leqslant 1$. We will use the shorthand notation $D_{n}$ for $\dim_{q}(u^{n})$ in the sequel.
\end{thm}

\begin{rem}
The following inequality always holds : $q+q^{-1} \geqslant N$.
\end{rem}

The representation theory of $U^{+}(F)$ was also explicitely computed by T. Banica in \cite{banica1997groupe}. However, we will only need the following result \cite[Thm. 1]{banica1997groupe} (see \cite{wang1995free} for the definition of the free product of discrete quantum groups).

\begin{thm}[Banica]\label{thm:freesubgroup}
The discrete quantum group $\F U^{+}(F)$ is a quantum subgroup of $\Z\ast \F O^{+}(F)$.
\end{thm}

The following lemma summarizes some standard calculations which will be used several times in the sequel.

\begin{lem}\label{lem:quantumdimension}
Let $a>b$ be integers, then $D_{a-b}^{-1}\leqslant D_{b}/D_{a}\leqslant q^{a-b}$. Moreover, for any integer $c$, $q^{c}D_{c}\leqslant (1-q^{2})^{-1}$.
\end{lem}

\begin{proof}
Let $n\in \Z$ such that $n\geqslant -b$. Decomposing $u^{b+n}\otimes u^{a+n+1}$ and $u^{b+n+1}\otimes u^{a+n}$ into sums of irreducible representations yields
\begin{equation*}
D_{b+n}D_{a+n+1} = D_{a-b+1} + \dots + D_{a+b+2n+1} \leqslant D_{a-b-1} + \dots + D_{a+b+2n+1} = D_{b+n+1}D_{a+n}
\end{equation*}
This inequality means that the sequence $(D_{b+n}/D_{a+n})_{n\geqslant -b}$ is increasing, thus any term is greater than its first term $D_{a-b}^{-1}$ and less than its limit $q^{a-b}$. The second part of the lemma is obvious since $q^{c}D_{c} = (1-q^{2c+2})/(1-q^{2})$.
\end{proof}

\section{Weak amenability for discrete quantum groups}\label{sec:quantumwa}

We now give some definitions and properties concerning weak amenability for discrete quantum groups. It is based on the notion of multipliers associated to bounded functions.

\begin{de}
Let $\h{\G}$ be a discrete quantum group and $a\in \ell^{\infty}(\h{\G})$. The \emph{left multiplier} associated to $a$ is the map $m_{a} : \Pol(\G) \rightarrow \Pol(\G)$ defined by
\begin{equation*}
(m_{a}\otimes \i)(u^{\alpha}) = (1\otimes ap_{\alpha})u^{\alpha},
\end{equation*}
for any irreducible representation $\alpha$ of $\G$.
\end{de}

\begin{rem}
This definition relies on the identification of $\ell^{\infty}(\h{\G})$ with $\prod \B(H_{\alpha})$ which is specific to the case of discrete quantum groups. However, since $W$ reads as $\prod u^{\alpha}$ in this identification, we can equivalently define the multiplier $m_{a}$ in the following way : for any $\omega\in \B(L^{2}(\G))_{*}$, $m_{a}((\i\otimes \omega)(W)) = (\i\otimes \omega)((1\otimes a)W)$. This definition makes sense in a more general setting and corresponds to the definition of J. Kraus and Z.J. Ruan in \cite{kraus1999approximation} for Kac algebras and to the definition of M. Junge, M. Neufang and Z.J. Ruan in \cite{junge2009representation} (see also \cite{daws2011multipliers}) for locally compact quantum groups.
\end{rem}

\begin{de}
A net $(a_{i})$ of elements of $\ell^{\infty}(\h{\G})$ is said to \emph{converge pointwise} to $a\in \ell^{\infty}(\h{\G})$ if $a_{i}p_{\alpha} \rightarrow ap_{\alpha}$ for any irreducible representation $\alpha$ of $\G$. An element $a\in \ell^{\infty}(\h{\G})$ is said to have \emph{finite support} if $ap_{\alpha}$ is non-zero only for a finite number of irreducible representations $\alpha$.
\end{de}

The key point to get a suitable definition of weak amenability is to have an intrinsic characterization of the completely bounded norm of a multiplier. Such a characterization is given by the following theorem \cite[Prop 4.1 and Thm 4.2]{daws2011multipliers}.

\begin{thm}[Daws]\label{theorem:quantumgilbert}
Let $\h{\G}$ be a discrete quantum group and $a\in \ell^{\infty}(\h{\G})$. Then $m_{a}$ extends to a competely bounded multiplier on $\B(L^{2}(\G))$ if and only if there exists a Hilbert space $K$ and two maps $\xi, \eta \in \B(L^{2}(\G), L^{2}(\G)\otimes K)$ such that $\|\xi\|\|\eta\| = \|m_{a}\|_{cb}$ and
\begin{equation}\label{eq:quantumgilbert}
(1\otimes \eta)^{*}\h{W}_{12}^{*}(1\otimes \xi)\h{W} = a\otimes 1.
\end{equation}
Moreover, we then have $m_{a}(x) = \eta^{*}(x\otimes 1)\xi$.
\end{thm}

Notice that thanks to this theorem, the completely bounded norm of $m_{a}$ is the same when it is extended to $C_{\text{red}}(\G)$, $L^{\infty}(\G)$ or $\B(L^{2}(\G))$. Denoting by $\|m_{a}\|_{cb}$ this norm, we can give a definition of weak amenability.

\begin{de}\label{de:quantumwa}
A discrete quantum group $\h{\G}$ is said to be \emph{weakly amenable} if there exists a net $(a_{i})$ of elements of $\ell^{\infty}(\h{\G})$ such that
\begin{itemize}
\item $a_{i}$ has finite support for all $i$.
\item $(a_{i})$ converges pointwise to $1$.
\item $K:=\limsup_{i} \|m_{a_{i}}\|_{cb}$ is finite.
\end{itemize}
The lower bound of the constants $K$ for all nets satisfying these properties is denoted $\Lambda_{cb}(\h{\G})$ and called the \emph{Cowling-Haagerup constant} of $\h{\G}$. By convention, $\Lambda_{cb}(\h{\G})=\infty$ if $\h{\G}$ is not weakly amenable.
\end{de}

It is clear on the definition that a discrete group $G$ is weakly amenable in the classical sense (see e.g. \cite[Def. 12.3.1]{brown2008finite}) if and only if the associated discrete quantum group is weakly amenable (and the constants are the same). We recall the following notions of weak amenability for operator algebras.

\begin{de}
A C*-algebra $A$ is said to be \emph{weakly amenable} if there exists a net $(T_{i})$ of linear maps from $A$ to itself such that
\begin{itemize}
\item $T_{i}$ has finite rank for all $i$.
\item $\|T_{i}(x)-x\|\rightarrow 0$ for all $x\in A$.
\item $K :=\limsup_{i}\|T_{i}\|_{cb}$ is finite.
\end{itemize}
The lower bound of the constants $K$ for all nets satisfying these properties is denoted $\Lambda_{cb}(A)$ and called the \emph{Cowling-Haagerup constant} of $A$. By convention, $\Lambda_{cb}(A) = \infty$ if the C*-algebra $A$ is not weakly amenable.

A von Neumann algebra $N$ is said to be \emph{weakly amenable} if there exists a net $(T_{i})$ of normal linear maps from $N$ to itself such that
\begin{itemize}
\item $T_{i}$ has finite rank for all $i$.
\item $(T_{i}(x)-x) \rightarrow 0$ ultraweakly for all $x\in N$.
\item $K :=\limsup_{i}\|T_{i}\|_{cb}$ is finite.
\end{itemize}
The lower bound of the constants $K$ for all nets satisfying these properties is denoted $\Lambda_{cb}(N)$ and called the \emph{Cowling-Haagerup constant} of $N$. By convention, $\Lambda_{cb}(N) = \infty$ if the von Neumann algebra $N$ is not weakly amenable.
\end{de}

See \cite[Thm. 5.14]{kraus1999approximation} for a proof of the following result.

\begin{thm}[Kraus-Ruan]\label{thm:quantumwa}
Let $\h{\G}$ be a \emph{unimodular} (i.e. the Haar state on $\G$ is tracial) discrete quantum group, then
\begin{equation*}
\Lambda_{cb}(\h{\G}) = \Lambda_{cb}(C_{\text{red}}(\G)) = \Lambda_{cb}(L^{\infty}(\G)).
\end{equation*}
\end{thm}

\section{Block decomposition and Haagerup inequality}\label{subsec:blocks}

Our aim is to prove a polynomial bound for the completely bounded norm of the projection on the linear span of coefficients of an irreducible representation $u^{d}$ in $C_{\text{red}}(O^{+}(F))$. Let us give some motivation for this. First note that this projection is simply the multiplier $m_{p_{d}}$ associated to $p_{d}\in \ell^{\infty}(\h{\G})$. If we choose for every integer $k$ and real number $t$ a scalar coefficient $b_{k}(t)$, we can define a net of (radial) elements
\begin{equation*}
a_{i}(t) = \sum_{k=0}^{i}b_{k}(t)p_{k}\in \ell^{\infty}(\h{\G}).
\end{equation*}
If the $b_{k}(t)$ have sufficiently nice properties and if the completely bounded norm of the operators $m_{p_{d}}$ can be controlled, the net $(a_{i}(t))$ will satisfy all the hypothesis in Definition \ref{de:quantumwa} and $\F O^{+}(F)$ will be weakly amenable.

Our strategy to obtain the polynomial bound is inspired from the proof of U. Haagerup's estimate for the completely bounded norm of projections on words of fixed length in free groups. The original proof is unpublished but the argument is detailed in G. Pisier's book \cite{pisier2003introduction}. Following the scheme of the proof, we will first, in this section, prove an operator-valued analogue of the Haagerup inequality.

From now on, we fix an integer $N > 2$ and a matrix $F\in GL_{N}(\C)$ satisfying $F\overline{F} \in \R.\Id$. We will write $\H$ for the Hilbert space  $L^{2}(O^{+}(F))$ which is identified to $\oplus_{k} \B(H_{k})$ as explained in Subsection \ref{subsec:gns} ($H_{k}$ being the carrier Hilbert space of the $k$-th irreducible representation). Let $H$ be a fixed Hilbert space and let $X\in \B(H)\odot \Pol(O^{+}(F))$ (it is enough to control the norm on this dense subalgebra), chose $d\in \N$ and set $X^{d} = (\i\otimes m_{p_{d}})(X)$. These objects should be thought of as "operator-valued functions with finite support" on $\F O^{+}(F)$. Our aim is to control the norm of $X^{d}$ using the norm of $X$.

\begin{rem}
Recall from \cite{vergnioux2007property} that there is a natural length function on $\Ir(O^{+}(F))$ such that the irreducible representation $u_{d}$ has length $d$. Using this notion, one could give a rigorous definition of "operator-valued functions with support in the words of length $d$". This, however, will not be needed here.
\end{rem}

We start by decomposing the operators into more elementary ones. For any two integers $a$ and $b$, we set
\begin{eqnarray*}
B_{a, b}(X) & := & (\i\otimes p_{a})X(\i\otimes p_{b}) \\
B_{a, b}(X^{d}) & := & (\i\otimes p_{a})X^{d}(\i\otimes p_{b})
\end{eqnarray*}
This is simply $X$ (resp. $X^{d}$) seen as an operator from $\B(H_{b})$ to $\B(H_{a})$ and obviously has norm less than $\|X\|$ (resp. $\|X^{d}\|$). We will call it a \emph{block}. The operator $X^{d}$ admits a particular decomposition with respect to these blocks.

\begin{lem}\label{lem:blockdecomposition}
Set $X^{d}_{j} = \displaystyle\sum_{k=0}^{+\infty}B_{d-j+k, j+k}(X^{d})$, then $X^{d} = \displaystyle\sum_{j=0}^{d}X^{d}_{j}$.
\end{lem}

\begin{proof}
Clearly, $X^{d} = \sum_{a, b}B_{a, b}(X^{d})$. If we decompose $X^{d}$ as $\sum_{i} T_{i}\otimes x_{i}$, with $x_{i}$ a coefficient of $u^{d}$ and $T_{i}\in \B(H)$, we see that $X^{d}$ sends $H\otimes (p_{b}\H)$ into $\oplus_{c} (H\otimes (p_{c}\H))$ where the sum runs over all irreducible subrepresentations $c$ of $d\otimes b$. Thus, we can deduce from Theorem \ref{thm:freefusion} that $B_{a, b}(X^{d})$ vanishes as soon as $a$ is not of the form $d+b-2j$ for some $0\leqslant j\leqslant \min(d, b)$. Consequently,
\begin{equation*}
X^{d} = \sum_{b=0}^{+\infty}\sum_{j=0}^{\min(d, b)}B_{d+b - 2j, b}(X^{d}) = \sum_{j=0}^{d}\sum_{b=j}^{+\infty}B_{d+b - 2j, b}(X^{d}).
\end{equation*}
Writing $b = k+j$, we get the desired result.
\end{proof}

This should be thought of as a decomposition according to the "number of deleted letters" in the action of $X^{d}$. Thanks to the triangle inequality, we can restrict ourselves to the study of $\|X^{d}_{j}\|$. Proposition \ref{prop:blocks} further reduces the problem to the study of only one specific block in $X^{d}_{j}$. Before stating and proving it, we have to introduce several notations and elementary facts.

Recall from Subsection \ref{subsec:gns} that for $\gamma\subset \alpha\otimes \beta$, $v_{\gamma}^{\alpha, \beta} : H_{\gamma} \mapsto H_{\alpha}\otimes H_{\beta}$ denotes an isometric intertwiner and let $M_{k}^{+} : \H \otimes \B(H_{k}) \rightarrow \H$ be the orthogonal sum of the operators $\Ad(v_{l+k}^{l, k})$. Under our identification of $\H$ with $\oplus \B(H_{k})$, the restriction of $M_{k}^{+}$ to $\B(H_{l})\otimes \B(H_{k})$ is just the map induced by the product composed with the orthogonal projection onto $\B(H_{l+k})$. If we endow $\B(H_{k})$ with the scalar product $\langle ., .\rangle_{k}$, it can be seen as a subspace of $\H$ and we can compute the norm of the restriction of $M_{k}^{+}$ to $\B(H_{l})\otimes \B(H_{k})$ with respect to the Hilbert structure on $\H \otimes \H$. Let $x\in \B(H_{l})\otimes \B(H_{k})$, then
\begin{eqnarray*}
\|M_{k}^{+}(x)\|^{2} & = & \frac{1}{D_{l+k}}\Tr(Q_{k+l}M_{k}^{+}(x)^{*}M_{k}^{+}(x)) \\
& = & \frac{1}{D_{l+k}}\Tr(Q_{k+l}(v^{l, k}_{l+k})^{*}x^{*}v^{l, k}_{l+k}(v^{l, k}_{l+k})^{*}x v^{l, k}_{l+k}) \\
& \leqslant & \frac{1}{D_{l+k}}\Tr(Q_{l+k}(v^{l, k}_{l+k})^{*}x^{*}x v^{l, k}_{l+k}) \\
& = & \frac{1}{D_{l+k}}\Tr(v^{l, k}_{l+k}Q_{l+k}(v^{l, k}_{l+k})^{*}x^{*}x) \\
& \leqslant & \frac{1}{D_{l+k}}(\Tr\otimes \Tr)((Q_{l}\otimes Q_{k})x^{*}x) \\
& = & \frac{D_{l}D_{k}}{D_{l+k}}\|x\|^{2}
\end{eqnarray*}
i.e. $\|M_{k}^{+}(p_{l}\otimes \i)\|^{2} = D_{l}D_{k}/D_{l+k}$ (the norm is attained at $x=v^{l, k}_{l+k}(v^{l, k}_{l+k})^{*}$).

\begin{rem}
Note that this computation also proves that $\|M_{k}^{+}\|^{2} = \displaystyle\frac{1-q^{2k+2}}{1-q^{2}}$ and in particular that $\|M_{1}^{+}\|^{2} = 1+q^{2}\leqslant 2$.
\end{rem}

\begin{rem}\label{rem:adjoint}
The computation of the adjoint of $M_{k}^{+}$ is similar to the computation of the norm. One has $(M_{k}^{+})^{*}p_{l+k} = (D_{l}D_{k}/D_{l+k})\Ad((v_{l+k}^{l, k})^{*})$.
\end{rem}

Let us now state and prove the main result of this section.

\begin{prop}\label{prop:blocks}
For integers $a, b$ and $c$, set
\begin{equation*}
N_{a, b}^{c} = 1 - \frac{D_{(a-b+c)/2}D_{(b-a+c)/2-1}}{D_{a+1}D_{b}}
\end{equation*}
whenever this expression makes sense. Then, if we set
\begin{equation*}
\chi_{j}^{d}(k) = \sqrt{\frac{D_{d-j}D_{j+k}}{D_{d-j+k}D_{j}}}\prod_{i=0}^{j-1}(N_{d-j+i, k+i}^{d-j+k})^{-1},
\end{equation*}
we have, for all $k$, $\|B_{d-j+k, j+k}(X^{d})\|\leqslant \chi_{j}^{d}(k)\|B_{d-j, j}(X^{d})\|$.
\end{prop}

\begin{proof}
Let us first focus on the one-dimensional case. Let $x$ be a coefficient of $u^{d}$ seen as an element of $\B(H_{d})$ and choose an integer $k$. Let us compare the two operators
\begin{eqnarray*}
A & = & [M_{k}^{+}(p_{d-j}xp_{j}\otimes \i) (M_{k}^{+})^{*}](\xi) \\
B & = & (p_{d-j+k}xp_{j+k})(\xi)
\end{eqnarray*}
for $\xi\in p_{j+k}\H = \B(H_{j+k})$. Setting $V = (\i\otimes v_{j+k}^{j, k})^{*}(v_{d-j}^{d, j}\otimes \i)v_{d-j+k}^{d-j, k}$, we have an intertwiner between $u^{d-j+k}$ and $u^{d\otimes (j+k)}$. Since that inclusion has multiplicity one, there is a complex number $\mu_{j}^{d}(k)$ such that 
\begin{equation*}
V = \mu_{j}^{d}(k)v_{d-j+k}^{d, j+k}.
\end{equation*}
Now, using Equation (\ref{eq:gns}) and Remark \ref{rem:adjoint}, we have
\begin{eqnarray*}
B & = & (v_{d-j+k}^{d, j+k})^{*}(x\otimes \xi)v_{d-j+k}^{d, j+k} \\
A & = & V^{*}(x\otimes \xi)\left(\frac{D_{j}D_{k}}{D_{j+k}}\right)V
\end{eqnarray*}
and consequently $B = \lambda^{d}_{j}(k) A$, with $\lambda_{j}^{d}(k) = (D_{j}D_{k}/D_{j+k})^{-1}\vert \mu_{j}^{d}(k) \vert^{-2}$. Let us compute $\vert\mu_{j}^{d}(k)\vert$. If we set $v_{+}^{a, b}=(v_{a+b}^{a, b})^{*}$ and define two morphisms of representations
\begin{eqnarray*}
\mathcal{T}_{A} & = & (v_{+}^{d-j, j}\otimes v_{+}^{j, 0} \otimes \i_{k})(\i_{d-j}\otimes t_{j}\otimes \i_{k}) v_{d-j+k}^{d-j, k} \\
\mathcal{T}_{B} & = & (v_{+}^{d-j, j}\otimes v_{+}^{j, k})(\i_{d-j}\otimes t_{j}\otimes \i_{k}) v_{d-j+k}^{d-j, k}
\end{eqnarray*}
we have, up to some complex numbers of modulus $1$,
\begin{equation*}
\mathcal{T}_{A} = \|\mathcal{T}_{A}\|(v_{d-j}^{d, j}\otimes \i)v_{d-j+k}^{d-j, k}\text{ and }\mathcal{T}_{B} = \|\mathcal{T}_{B}\|v_{d-j+k}^{d, j+k}.
\end{equation*}
Since moreover $(\i\otimes v_{j+k}^{j, k})^{*}\mathcal{T}_{A} = \mathcal{T}_{B}$, we get $\vert \mu_{j}^{d}(k)\vert^{2}=\|\mathcal{T}_{B}\|^{2}/\|\mathcal{T}_{A}\|^{2}$. Thanks to \cite[Prop. 2.3]{vergnioux2005orientation} and \cite[Lem. 4.8]{vergnioux2007property}, we can compute the norms of $\mathcal{T}_{A}$ and $\mathcal{T}_{B}$ and obtain
\begin{equation*}
\vert \mu_{j}^{d}(k) \vert^{2} = \prod_{i=0}^{j-1}\frac{N^{d-j+k}_{d-j+i, k+i}}{N^{d-j}_{d-j+i, i}} = \prod_{i=0}^{j-1}N^{d-j+k}_{d-j+i, k+i}.
\end{equation*}
Note that for $j=0$, the above product is not defined. However, $\lambda_{0}^{d}(k)=1$ since $\mathcal{T}_{A} = \mathcal{T}_{B}$ in that case. As $\lambda_{j}^{d}(k)$ does not depend on $\xi$, we have indeed proved the following equality in $\B(\H)$ :
\begin{equation*}
p_{d-j+k}xp_{j+k} =  \lambda_{j}^{d}(k)[M_{k}^{+}(p_{d-j}xp_{j}\otimes \i)(M_{k}^{+})^{*}].
\end{equation*}

Now we go back to the operator-valued case. We have $X^{d} = \sum_{i} T_{i}\otimes x_{i}$, where $x_{i}\in \Pol(O^{+}(F))$ is a coefficient of $u^{d}$ and $T_{i}\in \B(H)$, hence
\begin{equation*}
\lambda_{j}^{d}(k) [(\i\otimes M_{k}^{+})(B_{d-j, j}(X^{d})\otimes \i)(\i\otimes M_{k}^{+})^{*}] = B_{d-j+k, j+k}(X^{d}).
\end{equation*}
Using the norms of the restrictions of $M_{k}^{+}$ computed above, we get
\begin{equation*}
\|B_{d-j+k, j+k}(X^{d})\| \leqslant \lambda_{j}^{d}(k) \|(\i\otimes M_{k}^{+})B_{d-j, j}(X^{d})(\i\otimes M_{k}^{+})^{*}\| \leqslant \chi_{j}^{d}(k)\|B_{d-j, j}(X^{d})\|.
\end{equation*}
\end{proof}

\begin{cor}\label{cor:blocks}
There is a constant $K(q)$, depending only on $q$, such that for any $d\in \N$ and $0\leqslant j \leqslant d$, $\|X_{j}^{d}\|\leqslant K(q)\|B_{d-j, j}(X^{d})\|$.
\end{cor}

\begin{proof}
 According to Lemma \ref{lem:quantumdimension}, we have
\begin{equation*}
\frac{D_{d-j}D_{k-1}}{D_{d-j+i+1}D_{k+i}} \leqslant q^{i+1}q^{i+1} = q^{2i+2},
\end{equation*}
thus $(N_{d-j+i, k+i}^{d-j+k})^{-1}\leqslant (1-q^{2i+2})^{-1}$. Again by Lemma \ref{lem:quantumdimension}, $D_{d-j}/D_{d-j+k}\leqslant q^{k}$ and $D_{j+k}/D_{j}\leqslant D_{k}$, hence
\begin{equation*}
\chi_{j}^{d}(k)\leqslant \sqrt{q^{k}D_{k}}\prod_{i=0}^{j-1}\frac{1}{1-q^{2i+2}}\leqslant \frac{1}{\sqrt{1-q^{2}}}\prod_{i=0}^{+\infty}\frac{1}{1-q^{2i+2}} = K(q).
\end{equation*}
\end{proof}

As a summary of what has been worked out in this section, we can state an analogue of the Haagerup inequality for "operator-valued functions" on $\F O^{+}(F)$.

\begin{thm}
There exists a constant $K(q)$ depending only on $q$ such that
\begin{equation*}
\max_{0\leqslant j\leqslant d}\{\|B_{d-j, j}(X^{d})\|\} \leqslant \|X^{d}\| \leqslant K(q)(d+1)\max_{0\leqslant j\leqslant d}\{\|B_{d-j, j}(X^{d})\|\}
\end{equation*}
\end{thm}

\begin{proof}
The first inequality comes from the fact that $\|B_{d-j, j}(X^{d})\|\leqslant \|X^{d}\|$ and the second one from the triangle inequality combined with Proposition \ref{prop:blocks}.
\end{proof}

This inequality should be compared to A. Buchholz's inequality \cite[Thm 2.8]{buchholz1999norm} and to \cite[Eq 9.7.5]{pisier2003introduction} in the free group case.

\section{The completely bounded norm of projections}\label{subsec:recursion}

We now want to find some polynomial $P$ such that $\|X^{d}\|\leqslant P(d)\|X\|$. Thanks to Proposition \ref{prop:blocks}, the problem reduces to finding a polynomial $Q$ such that $\|B_{d-j, j}(X^{d})\|\leqslant Q(d)\|X\|$. This will be done using the following recursion formula.

\begin{prop}\label{prop:recursion}
Set
\begin{equation*}
N_{1}^{+} = \bigoplus_{l}\frac{D_{l+1}}{D_{1}D_{l}}M_{1}^{+}(p_{l}\otimes \i).
\end{equation*}
According to Remark \ref{rem:adjoint}, $(N_{1}^{+})^{*}$ is the sum of the operators $\Ad((v_{l+k}^{l, k})^{*})$. There are coefficients $C_{j}^{d}(s)$ such that for $0\leqslant j\leqslant d$,
\begin{eqnarray*}
B_{d-j+1, j+1}(X) & - & (\i\otimes M_{1}^{+})(B_{d-j, j}(X)\otimes \i)(\i\otimes N_{1}^{+})^{*} \\
= B_{d-j+1, j+1}(X^{d+2}) & + & \sum_{s=0}^{\min(j, d-j)}C_{j}^{d}(s)B_{d-j+1, j+1}(X^{d-2s})
\end{eqnarray*} 
\end{prop}

\begin{proof}
The idea of the proof is similar to the one used in the proof of Proposition \ref{prop:blocks}. We first consider the one-dimensional case. Let $x$ be a coefficient of $u^{l}$ seen as an element of $\B(H_{l})$. Fix an element $\xi\in p_{j+1}\H$. Again, the operators
\begin{eqnarray*}
A & = & [M_{1}^{+}(p_{d-j}xp_{j}\otimes \i)(N_{1}^{+})^{*}](\xi) \\
B & = & (p_{d-j+1}x_{l}p_{j+1})(\xi)
\end{eqnarray*}
are proportional. Note that if $l>d+2$, $l<\vert d-2j\vert$ or $l-d$ is not even, both operators are $0$. Note also that if $l=d+2$, $A = 0$. The other values of $l$ can be written $d-2s$ for some positive integer $s$ between $0$ and $\min(j, d-j)$. In that case, the existence of a scalar $\nu_{j}^{d}(s)$ such that $B = \nu^{d}_{j}(s) A$ follows from the same argument as in the proof of Proposition \ref{prop:blocks}. Let us compute $\nu^{d}_{j}(s)$, noticing that thanks to the normalization of $N_{1}^{+}$, the constant $\nu_{j}^{d}(s)$ only corresponds to the "$\mu$-part" of the constant $\lambda$ of Proposition \ref{prop:blocks}. This time we have to set
\begin{eqnarray*}
\mathcal{T}_{A} & = & (v_{+}^{d-s-j, j-s}\otimes v_{+}^{j-s, s}\otimes \i_{1})(\i_{d-j-s}\otimes t_{j-s}\otimes \i_{s+1}) v_{d-j+1}^{d-j-s, s+1} \\
\mathcal{T}_{B} & = & (v_{+}^{d-s-j, j-s}\otimes v_{+}^{j-s, s+1})(\i_{d-j-s}\otimes t_{j-s}\otimes \i_{s+1}) v_{d-j+1}^{d-j-s, s+1}
\end{eqnarray*}
Again, applying \cite[Prop. 2.3]{vergnioux2005orientation} and \cite[Lem. 4.8]{vergnioux2007property} yields
\begin{equation*}
\nu^{d}_{j}(s) = \frac{\|\mathcal{T}_{A}\|^{2}}{\|\mathcal{T}_{B}\|^{2}} = \prod_{i=0}^{j-s-1}\frac{N^{d-j}_{d-j-s+i, s+i}}{N^{d-j+1}_{d-j-s+i, s+i+1}}.
\end{equation*}

Like in the proof of Proposition \ref{prop:blocks}, we can now go back to the operator-valued case. We have
\begin{equation*}
X = \sum_{l}\sum_{i=0}^{k(l)} T_{l}^{(i)}\otimes x_{l}^{(i)}
\end{equation*}
where $x_{l}^{(i)}\in \Pol(O^{+}(F))$ are coefficients of $u^{l}$ and $T_{l}^{(i)}\in \B(H)$. Setting
\begin{equation*}
X^{l} = \sum_{i=0}^{k(l)}T_{l}^{(i)}\otimes x_{l}^{(i)},
\end{equation*}
we have
\begin{equation*}
B_{d-j+1, j+1}(X^{l}) = \nu^{d}_{j}(s) (\i\otimes M_{1}^{+})(B_{d-j, j}(X^{l})\otimes \i)(\i\otimes N_{1}^{+})^{*}
\end{equation*}
and setting $C_{j}^{d}(s) = 1 - \nu_{j}^{d}(s)^{-1}$ concludes the proof.
\end{proof}

The last result we need is a control on the coefficients $C_{j}^{d}(s)$ and $\chi^{d}_{j}(s)$.

\begin{lem}\label{lem:constantsum}
For any $0\leqslant j\leqslant d$, $\displaystyle\sum_{s=0}^{\min(j, d-j)}\vert C_{j}^{d}(s)\vert\chi^{d-2s}_{j-s}(s+1) \leqslant 1$.
\end{lem}

\begin{proof}
We first give another expression of $\vert C_{j}^{d}(s)\vert$. Decomposing into sums of irreducible representations yields
\begin{eqnarray*}
D_{d-s-j+i+1}D_{s+i+1} - D_{d-s-j}D_{s} = D_{d-j+2} + \dots + D_{d-j+2i+2} = D_{i}D_{d-j+i+2} \\
D_{d-s-j+i+1}D_{s+i} - D_{d-s-j}D_{s-1} = D_{d-j+1} + \dots + D_{d-j+2i+1} = D_{i}D_{d-j+i+1}
\end{eqnarray*}
which implies that
\begin{equation*}
N^{d-j+1}_{d-j-s+i, s+i+1} = \frac{D_{i}D_{d-j+i+2}}{D_{d-s-j+i+1}D_{s+i+1}}\text{ and }N^{d-j}_{d-s-j+i, s+i} = \frac{D_{i}D_{d-j+i+1}}{D_{d-s-j+i+1}D_{s+i}}.
\end{equation*}
Hence
\begin{equation*}
\nu_{j}^{d}(s) = \prod_{i=0}^{j-s-1}\frac{N^{d-j}_{d-s-j+i, s+i}}{N^{d-j+1}_{d-j-s+i, s+i+1}} = \prod_{i=0}^{j-s-1}\frac{D_{d-j+i+1}D_{s+i+1}}{D_{s+i}D_{d-j+i+2}} = \frac{D_{j}D_{d-j+1}}{D_{s}D_{d-s+1}}.
\end{equation*}
Again, noticing that $D_{j}D_{d-j+1} - D_{s}D_{d-s+1} = D_{d-j-s}D_{j-s-1}$ yields
\begin{equation*}
\vert C_{j}^{d}(s)\vert = \vert 1 - \nu_{j}^{d}(s)^{-1}\vert = \frac{D_{d-j-s}D_{j-s-1}}{D_{d-j+1}D_{j}}.
\end{equation*}
According to Lemma \ref{lem:quantumdimension}, we thus have
\begin{equation*}
\vert C_{j}^{d}(s)\vert \leqslant q^{s+1}q^{s+1} = q^{2s+2}
\end{equation*}
Now we turn to $\chi_{j-s}^{d-2s}(s+1)$. In fact, we are going to bound $\chi_{j}^{d}(s+1)$ independantly of $d$ and $j$. Decomposing into sums of irreducible representations, we get
\begin{equation*}
D_{d-j+i+1}D_{k+i} - D_{d-j}D_{k-1} = D_{d-j+k+1} + \dots + D_{d-j+k+2i+1} = D_{i}D_{d-j+k+i+1},
\end{equation*}
which implies that $N_{d-j+i, k+i}^{d-j+k} = D_{i}D_{d-j+k+i+1}/D_{d-j+i+1}D_{k+i}$. Now we can compute
\begin{eqnarray*}
\frac{\chi_{j}^{d}(s+1)}{\chi_{j}^{d}(s)} & = & \sqrt{\frac{D_{j+s+1}D_{d-j+s}}{D_{j+s}D_{d-j+s+1}}}\prod_{i=0}^{j-1}\frac{D_{s+1+i}D_{d-j+s+i+1}}{D_{s+i}D_{d-j+s+i+2}} \\
& = & \sqrt{\frac{D_{j+s+1}D_{d-j+s}}{D_{j+s}D_{d-j+s+1}}}\frac{D_{j+s}D_{d-j+s+1}}{D_{s}D_{d+s+1}} \\
& = & \frac{\sqrt{D_{j+s}D_{d-j+s+1}D_{d-j+s}D_{j+s+1}}}{D_{s}D_{d+s+1}}.
\end{eqnarray*}
Using Lemma \ref{lem:quantumdimension} again, we get
\begin{equation*}
\frac{\chi_{j}^{d}(s+1)}{\chi_{j}^{d}(s)} \leqslant \sqrt{q^{j}D_{j}q^{d-j}D_{d-j}} \leqslant \frac{1}{1-q^{2}}.
\end{equation*}
Since $\chi_{j}^{d}(1) \leqslant (1-q^{2})^{-1}$, we have proved that $\chi_{j}^{d}(s+1)\leqslant (1-q^{2})^{-s-1}$. This bound is independant of $d$ and $j$, thus it also works for $\chi_{j-s}^{d-2s}(s+1)$. Combining this with our previous estimate we can compute
\begin{equation*}
\sum_{s=0}^{\min(j, d-j)}\vert C_{j}^{d}(s)\vert\chi_{j-s}^{d-2s}(s+1) \leqslant \sum_{s=0}^{+\infty}\left(\frac{q^{2}}{1-q^{2}}\right)^{s+1} = \frac{q^{2}}{1-2q^{2}}.
\end{equation*}
The last term is less than $1$ as soon as $q\leqslant 1/\sqrt{3}$, hence in particular for any $q$ such that $q+q^{-1}\geqslant 3$.
\end{proof}

Gathering all our results will now give the estimate we need. To make things more clear, we will proceed in two steps. First we bound the norms of the blocks of $X^{d}$.

\begin{prop}\label{prop:polynomialblocks}
There exists a polynomial $Q$ such that for any integer $d$ and $0\leqslant j\leqslant d$, $\|B_{d-j, j}(X^{d})\| \leqslant Q(d)\|X\|$.
\end{prop}

\begin{proof}
First note that $B_{d, 0}(X^{d}) = B_{d, 0}(X)$ and $B_{0, d}(X^{d}) = B_{0, d}(X)$, hence we only have to consider the case $1\leqslant j \leqslant d-1$. Moreover, applying the triangle inequality to the recursion relation of Proposition \ref{prop:recursion} yields
\begin{eqnarray*}
\|B_{d-j+1, j+1}(X^{d+2})\| & \leqslant & (1+\|M_{1}^{+}\|\|N_{1}^{+}\|)\|X\| \\
& + & \sum_{s=0}^{\min(j, d-j)}\vert C_{j}^{d}(s)\vert \|B_{d-j+1, j+1}(X^{d-2s})\|.
\end{eqnarray*}
We proceed by induction, with the following induction hypothesis

$H(d)$ : "For any integer $l \leqslant d$ and any $0\leqslant j\leqslant l$, $\|B_{l-j, j}(X^{l})\| \leqslant Q(l)\|X\|$ with $Q(X) = 2X+1$."

Because of the remark at the beginning of the proof, $H(0)$ and $H(1)$ are true. Knowing this, we just have to prove that for any $d$, $H(d)$ implies the inequality for $d+2$. Indeed, this will prove that assuming $H(d)$, both the inequalities for $d+1$ (noticing that $H(d)$ implies $H(d-1)$) and $d+2$ are true, hence $H(d+2)$ will hold.

Assume $H(d)$ to be true for some $d$ and apply the recursion formula above. The blocks in the right-hand side of the inequality are of the form $B_{d-j+1, j+1}(X^{d-2s})$. By Proposition \ref{prop:blocks} and $H(d)$,
\begin{eqnarray*}
\|B_{d-j+1, j+1}(X^{d-2s})\| & = & \|B_{(d-2s)-(j-s)+s+1, (j-s)+s+1}(X^{d-2s})\| \\
& \leqslant & \chi_{j-s}^{d-2s}(s+1)\|B_{(d-2s) - (j-s), (j-s)}(X^{d-2s})\| \\
& \leqslant & \chi_{j-s}^{d-2s}(s+1)Q(d-2s)\|X\|.
\end{eqnarray*}
Then, bounding $Q(d-2s)$ by $Q(d)$ and using Lemma \ref{lem:constantsum} yields
\begin{equation*}
\|B_{d-j+1, j+1}(X^{d+2})\| \leqslant 3\|X\| + Q(d)\|X\| \leqslant Q(d+2)\|X\|.
\end{equation*}
Since $\|B_{d-j+1, j+1}(X^{d+2})\| = \|B_{(d+2)-(j+1), j+1}(X^{d+2})\|$, the inequality is proved for $1\leqslant j+1\leqslant d+1$. In other words, we have $\|B_{d-J, J}(X^{d+2})\| \leqslant Q(d+2)\|X\|$ for any $1\leqslant J\leqslant d+1$. As noted at the beginning of the proof, this is enough to get $H(d+2)$.
\end{proof}

Secondly we bound the norm of $X^{d}$ itself.

\begin{thm}\label{thm:completelyboundedprojection}
Let $F\in GL_{N}(\C)$ be such that $F\overline{F}\in \R.\Id$ and let $0\leqslant q\leqslant 1$ be the real number defined in Theorem \ref{thm:freefusion}. Then, if $q\leqslant 3^{-1/2}$ (in particular if $N\geqslant 3$), there exists a polynomial $P$ such that for all integers $d$,
\begin{equation*}
\|m_{p_{d}}\|_{cb} \leqslant P(d).
\end{equation*}
\end{thm}

\begin{proof}
We use the notations of Proposition \ref{prop:polynomialblocks}. We know from Corollary \ref{cor:blocks} that $\|X^{d}_{j}\| \leqslant K(q)\|B_{d-j, j}(X^{d})\|$, thus $\|X^{d}_{j}\|\leqslant K(q)Q(d)\|X\|$. If we set $P(X) = K(q)(X+1)Q(X)$, we get $\|X^{d}\|\leqslant P(d)\|X\|$ by applying the triangle inequality to the decomposition of Lemma \ref{lem:blockdecomposition}.
\end{proof}

\begin{rem}
One could slightly improve this bound by noticing that since we can replace $Q(d)$ by $1$ when $j=0$ or $d$, $\|X^{d}\|\leqslant K(q)(2d^{2}-d+1)\|X\|$.
\end{rem}

\begin{rem}
When $q=1$, we get the usual compact group $SU(2)$. It was proved in \cite{vergnioux2007property} that this group (or rather its discrete quantum dual) has the Rapid Decay property, and that consequently $m_{p_{d}}$ grows at most polynomially. Since any bounded map from a C*-algebra into a \emph{commutative} C*-algebra is completely bounded with same norm (e.g. by \cite[Prop 34.6]{conway2000course}), Theorem \ref{thm:completelyboundedprojection} also works in that case. This of course suggests that it holds for $SU_{q}(2)$ for any value of $q$, though the majorizations of Lemma \ref{lem:constantsum} are not good enough to provide such a statement.
\end{rem}

It is proved in \cite[Thm 9.7.4]{pisier2003introduction} that in the free group case, the completely bounded norm of the projections on words of fixed length grows at most linearly. Our technique cannot determine whether such a result still holds in the quantum case but proves the slightly weaker fact that the growth is at most quadratic. However, we can prove that it is also at least linear. Let us first recall that the sequence $(\mu_{k})$ of (dilated) Chebyshev polynomials of the second kind is defined by $\mu_{0}(X) = 1$, $\mu_{1}(X) = X$ and
\begin{equation*}
X\mu_{k}(X) = \mu_{k-1}(X) + \mu_{k+1}(X)
\end{equation*}

\begin{prop}
Let $F\in GL_{N}(\C)$ be such that $F\overline{F}\in \R.\Id$. Then, there exists a polynomial $R$ of degree one such that
\begin{equation*}
\|m_{p_{d}}\|_{cb}\geqslant R(d).
\end{equation*}
\end{prop}

\begin{proof}
Since $\|m_{p_{d}}\|_{cb} \geqslant \|m_{p_{d}}\|$, we will simply prove a lower bound for this second norm. Let $\chi_{n}\in \Pol(\G)$ be the character of the representation $u^{n}$, i.e.
\begin{equation*}
\chi_{n} = (\i\otimes \Tr)(u_{n}).
\end{equation*}
Our aim is to prove that looking at the action of $m_{p_{d}}$ on $\chi_{d+2} - \chi_{d}$ is enough to get the lower bound.

It is known (see \cite{banica1996theorie}) that sending $\chi_{n}$ to the restriction to $[-2, 2]$ of $\mu_{n}$ yields an isomorphism between the sub-C*-algebra of $C_{\text{red}}(O^{+}(F))$ generated by the elements $\chi_{n}$ and $C([-2, 2])$. Moreover, the restriction of these polynomials to the interval $[-2, 2]$ form a Hilbert basis with respect to the scalar product associated to the semicircular law
\begin{equation*}
d\nu = \frac{\sqrt{4-t^{2}}}{2\pi}dt.
\end{equation*}
Let us denote by $\pi :C([-2, 2]) \rightarrow \B(L^{2}([-2, 2], d\nu))$ the faithful representation by multiplication operators. What precedes means precisely that we have, for any finitely supported sequence $(a_{n})$,
\begin{equation*}
\left\|\sum_{n} a_{n}\chi_{n}\right\|_{C_{\text{red}}(O_{N}^{+})} = \left\|\sum_{n} a_{n}\mu_{n\vert[-2, 2]}\right\|_{\infty} = \left\|\sum_{n} a_{n}\pi(\mu_{n})\right\|_{\B(L^{2}([-2, 2], d\nu))}.
\end{equation*}
Let $e_{i}$ denote the image of $\mu_{i}$ in $L^{2}([-2, 2], d\nu)$ and denote by $T_{n}$ the operator sending $e_{i}$ to $e_{i+n}$ for $n\geqslant 0$. Letting $E_{j}$ denote the linear span of the vectors $e_{i}$ for $0\leqslant i\leqslant j$, we can also define operators $T_{-n}$ which are $0$ on $E_{n-1}$ and send $e_{i}$ to $e_{i-n}$ for $i\geqslant n$. The last operator we need, denoted $S_{n}$, sends $e_{i}\in E_{n}$ to $e_{n-i}$ and is $0$ on $E_{n}^{\perp}$. These translation operators obviously have norm $1$. Moreover, a simple computation using Theorem \ref{thm:freefusion} (or equivalently the recursion relation of the Chebyshev polynomials) shows that
\begin{equation*}
\pi(\mu_{n+2} - \mu_{n}) = T_{n+2} - S_{n} - T_{-(n+2)}.
\end{equation*}
Thus $\|\chi_{n+2} - \chi_{n}\| = \|\pi(\mu_{n+2} - \mu_{n})\|\leqslant 3$. On the other hand, it easily seen that $\mu_{n}(2) = n+1$. In fact, this is true for $\mu_{1}(X) = X$ and $\mu_{2}(X) = X^{2}-1$ and we have the recursion relation
\begin{equation*}
2\mu_{n}(2) = \mu_{n+1}(2) + \mu_{n-1}(2).
\end{equation*}
This implies that $\|\chi_{n}\| = \|\mu_{n}\|_{\infty} \geqslant n+1$. Combining these two facts yields
\begin{equation*}
\|m_{p_{d}}\|\geqslant \frac{\|m_{p_{d}}(\chi_{d+2} - \chi_{d})\|}{\|\chi_{d+2} - \chi_{d}\|} = \frac{\|-\chi_{d}\|}{\|\chi_{d+2} - \chi_{d}\|} \geqslant \frac{d+1}{3}
\end{equation*}
and setting $R(X) = (X+1)/3$ concludes the proof.
\end{proof}

\section{Monoidal equivalence and weak amenability}\label{subsec:final}

All the results proved so far hold in great generality, i.e. at least for any $\F O^{+}(F)$ with $F$ of size at least $3$ satisfying $F\overline{F}\in \R.\Id$. However, we will need in the proof of Theorem \ref{thm:maintheorem} a result of M. Brannan proving that some specific multipliers are completely positive. That assertion has up to now only been proved in the case $F = I_{N}$, hence our restriction in Theorem \ref{thm:maintheorem}. We will discuss this issue later on. Let us first deduce weak amenability of free quantum groups from the preceding sections.

\begin{thm}\label{thm:maintheorem}
Let $N\geqslant 2$ be an integer, then the discrete quantum groups $\F O_{N}^{+}$ and $\F U_{N}^{+}$ are weakly amenable and their Cowling-Haagerup constant is equal to $1$.
\end{thm}

\begin{proof}
For $N=2$, this result is already known by amenability of the discrete quantum group $\F O_{2}^{+} = \widehat{SU_{-1}(2)}$. Thus, we will assume $N>2$. We are going to use a net of elements in $\ell^{\infty}(\F O_{N}^{+})$ introduced by M. Brannan in \cite{brannan2011approximation} to prove the Haagerup property and the metric approximation property. For $t\in [0, N]$, set $b_{k}(t) = \mu_{k}(t)/\mu_{k}(N)$ and
\begin{equation*}
a_{i}(t) = \sum_{k=0}^{i} b_{k}(t)p_{k}\in \ell^{\infty}(\F O_{N}^{+}).
\end{equation*}
This is a net of finite rank elements converging pointwise to the identity and we now have to prove that the completely bounded norms of the associated multipliers satisfy the boundedness condition. If we fix some $2 < t_{0} < 3$, then \cite[Prop. 4.4]{brannan2011approximation} asserts the existence of a constant $K_{0}$, depending only on $t_{0}$, such that for any $t_{0}\leqslant t < N$, $0<b_{k}(t)<K_{0}(t/N)^{k}$. According again to \cite[Prop. 4.4]{brannan2011approximation}, the multipliers associated to the elements $a(t) = \sum_{k} b_{k}(t)p_{k}$ (where the sum runs from $0$ to infinity) are unital and completely positive. Moreover, for any $t_{0}\leqslant t < N$,
\begin{eqnarray*}
\|m_{a(t)} - m_{a_{i}(t)}\|_{cb} & \leqslant & \sum_{k > i} K_{0}\left(\frac{t}{N}\right)^{k}\|m_{p_{k}}\|_{cb}. \\
\end{eqnarray*}
This sum tends to $0$ as $i$ goes to infinity since Theorem \ref{thm:completelyboundedprojection} implies that it is the rest of an absolutely converging series. This implies that $\limsup \|m_{a_{i}(t)}\|_{cb} = 1$. In other words, $\Lambda_{cb}(\F O_{N}^{+}) = 1$. By \cite[Thm. 4.2]{freslon2012note}, we also have $\Lambda_{cb}(\Z\ast \F O_{N}^{+}) = 1$, hence $\Lambda_{cb}(\F U_{N}^{+}) = 1$ by Theorem \ref{thm:freesubgroup}.
\end{proof}

\begin{rem}
Let us point out that the above results, \cite[Thm. 4.2]{freslon2012note} and the isomorphisms of \cite[Prop 3.2]{weber2012classification} (see also \cite[Thm 4.1]{raum2012isomorphisms}) imply that the free bistochastic quantum groups $B_{N}^{+}$ and their symmetrized versions $(B_{N}^{+})'$ and $(B_{N}^{+})^{\sharp}$ have the Haagerup property and are weakly amenable with Cowling-Haagerup constant equal to $1$.
\end{rem}

Let us now explain how this technique can be extended to other families of discrete quantum groups. We will use the notion of monoidal equivalence and we refer the reader to \cite{bichon2006ergodic} for the relevant definitions and properties. We are thankful to S. Vaes for suggesting the following argument.

\begin{prop}\label{prop:monoidalequivalence}
Let $\G_{1}$ and $\G_{2}$ be two monoidally equivalent compact quantum groups. Then, if we index the equivalence classes of irreducible representations of both quantum groups by the same set through the monoidal equivalence, and if we denote by $p^{i}_{x}$ the projection in $\ell^{\infty}(\h{\G}_{i})$ corresponding to $x\in \Ir(\G_{i})$ for $i=1, 2$, we have
\begin{equation*}
\|m_{p^{1}_{x}}\|_{cb} = \|m_{p^{2}_{x}}\|_{cb}.
\end{equation*}
\end{prop}

\begin{proof}
Let $B$ be the linking algebra given by \cite[Thm 3.9]{bichon2006ergodic}. There is an action $\alpha : C(\G_{1})\rightarrow C(\G_{1})\otimes B$ such that
\begin{equation*}
(m_{p^{1}_{x}}\otimes \i)\circ \alpha = \alpha\circ Q_{x},
\end{equation*}
where $Q_{x}$ denotes the projection in $B$ onto the spectral subspace associated to the irreducible representation $u^{x}$. The injective $*$-homomorphism $\alpha$ being completely isometric, we deduce
\begin{equation*}
\|Q_{x}\|_{cb} \leqslant \|m_{p^{1}_{x}}\|_{cb}.
\end{equation*}
Now, we know from the proof of \cite[Thm 6.1]{vaes2007boundary} that there is an injective $*$-homomorphism $\theta : C(\G_{2})\rightarrow B\otimes B^{op}$ such that
\begin{equation*}
(Q_{x}\otimes \i)\circ\theta = \theta\circ m_{p^{2}_{x}},
\end{equation*}
yielding
\begin{equation*}
\|m_{p^{2}_{x}}\|_{cb} \leqslant \|Q_{x}\|_{cb}.
\end{equation*}
\end{proof}

Proposition \ref{prop:monoidalequivalence} means in particular that proving the polynomial bound in the case of $SU_{q}(2)$ (say at least for $q\leqslant 3^{-1/2}$) would give an alternative proof of Theorem \ref{thm:completelyboundedprojection}. However, it is not clear to us that such a computation would really be easier. We can now give a second class of examples of weakly amenable discrete quantum groups. We refer the reader to \cite{banica1999symmetries} and \cite{banica2002quantum} for the definition of quantum automorphism groups.

\begin{thm}
Let $B$ be a finite-dimensional C*-algebra with $\dim(B)\geqslant 6$ and let $\sigma$ be the $\delta$-trace on $B$. Then, the compact quantum automorphism group of $(B, \sigma)$ is weakly amenable and has Cowling-Haagerup constant equal to $1$.
\end{thm}

\begin{proof}
Let us first prove a more general statement. Consider the sub-C*-algebra $C(SO_{q}(3))$ of $C(SU_{q}(2))$ generated by the coefficients of $u^{\otimes 2}$, where $u$ denotes the fundamental representation of $SU_{q}(2)$. The restriction of the coproduct turns this C*-algebra into a compact quantum group, called $SO_{q}(3)$, which can be identified with the compact quantum automorphism group of $M_{2}(\C)$ with respect to a $(q+q^{-1})$-form (see \cite{soltan2010quantum}). Its irreducible representations can be identified with the even irreducible representations of $SU_{q}(2)$, and re-indexing them by $\N$ gives the $SO(3)$-fusion rules $u^{1}\otimes u^{n} = u^{n-1}\oplus u^{n}\oplus u^{n+1}$. Consequently, we have $\|m_{p_{d}}\|_{cb}\leqslant P(2d)$ as soon as $q\leqslant 3^{-1/2}$. We know from \cite[Thm 4.7]{de2010actions} that if $\sigma$ is any $\delta$-form on $B$, the compact quantum automorphism group of $(B, \sigma)$ is monoidally equivalent to $SO_{q}(3)$ if and only if $q+q^{-1} = \delta$. Thus, by Proposition \ref{prop:monoidalequivalence}, we have
\begin{equation*}
\|m_{p_{d}}\|_{cb}\leqslant P(2d)
\end{equation*}
for any integer $d$ as soon as $\delta$ is big enough. Direct computation shows that $\delta \geqslant \sqrt{6}$ is a sufficient condition.
If now $\sigma$ is the $\delta$-trace, $\delta = \sqrt{\dim(B)}$ and it was proved in \cite[Thm 4.2]{brannan2012reduced} that the elements
\begin{equation*}
a(t) = \sum_{k=0}^{+\infty} \frac{\mu_{2k}(\sqrt{t})}{\mu_{2k}(\sqrt{\dim(B)})}p_{k}
\end{equation*}
implement the Haagerup property. Applying the same proof as in Theorem \ref{thm:maintheorem} then yields weak amenability with Cowling-Haagerup constant $1$.
\end{proof}

\begin{rem}
A particular case of the previous theorem is the \emph{quantum permutation groups} $S_{N}^{+}$ (for $N\geqslant 6$) defined by S. Wang in \cite{wang1998quantum}. We can also deduce the Haagerup property and weak amenability with Cowling-Haagerup constant equal to $1$ for its symetrized version $(S_{N}^{+})'$.
\end{rem}

\begin{rem}
If $\dim(B)\leqslant 4$, the quantum automorphism group of $(B, \sigma)$ is amenable and is therefore weakly amenable with Cowling-Haagerup constant equal to $1$. Hence, the only case which is not covered by the previous theorem is the case of quantum automorphism groups of five-dimensional C*-algebras. There are two such C*-algebras, namely $\C^{5}$ and $M_{2}(\C)\oplus\C$. The quantum automorphism groups of these spaces are known to have the Haagerup property and we of course believe that they are weakly amenable with Cowling-Haagerup constant equal to $1$, though we do not have a proof of this fact.
\end{rem}

Let us further comment the consequences of Proposition \ref{prop:monoidalequivalence}. Let us say that an element $a\in \ell^{\infty}(\h{\G})$ is \emph{central} if it is of the form
\begin{equation*}
a = \sum_{\alpha\in \Ir(\G)} b_{\alpha}p_{\alpha}
\end{equation*}
where $b_{\alpha}\in \C$ (i.e. $a$ belongs to the centre of $\ell^{\infty}(\h{\G})$). Making linear combinations in the proof of Proposition \ref{prop:monoidalequivalence}, we see that if $\G_{1}$ and $\G_{2}$ are monoidally equivalent compact quantum groups, then any central element $a$ in $\h{\G}_{1}$ gives rise to a central element $a'$ in $\h{\G}_{2}$, the multipliers of which have the same completely bounded norm. Thus, weak amenability transfers through monoidal equivalence as soon as it can be implemented by central elements. Assume moreover that $m_{a}$ is u.c.p., then $m_{a'}$ is unital and has completely bounded norm $1$. Since any unital linear map of norm $1$ between two C*-algebras is positive (see e.g. \cite[Prop 33.9]{conway2000course}), we can conlude that $m_{a'}$ is also u.c.p. This gives us the first examples of non-amenable, non-unimodular discrete quantum groups having approximation properties.

\begin{prop}\label{prop:monoidalequivalencebis}
Let $\G$ be compact quantum group which is monoidally equivalent to $O^{+}_{N}$ or $U_{N}^{+}$ for some $N$ or to the compact quantum automorphism group of $(B, \sigma)$ for some finite-dimensional C*-algebra $B$ of dimension at least $6$ endowed with its $\delta$-trace $\sigma$, then $\h{\G}$ has the Haagerup property and is weakly amenable with Cowling-Haagerup constant equal to $1$.
\end{prop}

\begin{rem}
Using these arguments we can in fact recover \cite[Thm 4.2]{brannan2012reduced} directly from \cite[Thm 4.5]{brannan2011approximation}.
\end{rem}

\begin{rem}
Note that under the conditions of Proposition \ref{prop:monoidalequivalencebis}, the linking algebra $B$ giving the monoidal equivalence also has the Haagerup property relative to the unique invariant state and is weakly amenable with Cowling-Haagerup constant $1$.
\end{rem}

We end with some comments on the following very natural question : when is it possible to implement an approximation property by multipliers associated to central elements on a discrete quantum group ?

Let us make this question more formal. Let $A$ be an approximation property and say that a discrete quantum group has \emph{central A} if there are central multipliers implementing the property $A$. On the positive part of the problem, we have the following obvious facts :
\begin{enumerate}
\item If $\h{\G}$ has central $A$, then it has $A$ (the converse is true for any discrete group since any element is central).
\item If $\h{\G}_{1}$ has central $A$ and if $\G_{1}$ is monoidally equivalent to $\G_{2}$, then $\h{\G}_{2}$ has central $A$.
\item In the previous case, the linking algebra also has central $A$ (with respect to the projections on spectral subspaces).
\item If $\h{\G}$ has the Haagerup property and is unimodular, then it has the central Haagerup property (this is a consequence of M. Brannan's averaging technique used in the proof of \cite[Thm 3.7]{brannan2011approximation}).
\end{enumerate} 

On the negative part of the problem, we can make two remarks :
\begin{enumerate}
\item Amenability of $SU_{q}(2)$ (or $SO_{q}(3)$) cannot be implemented by central multipliers (whereas there is an obvious way to implement it with general multipliers), otherwise all the compact quantum groups $O^{+}(F)$ (or all the compact quantum automorphism groups of finite-dimensional C*-algebras with respect to some $\delta$-form) would be amenable.
\item If $q+q^{-1}$ is an integer, $SU_{q}(2)$ and $SO_{q}(3)$ have the central Haagerup property and are centrally weakly amenable with constant $1$. Hence these central approximation properties cannot be deduced from amenability.
\end{enumerate}

\bibliographystyle{amsplain}
\bibliography{../quantum}

\end{document}